\documentclass[reqno]{amsart}

\usepackage[T1]{fontenc}
\usepackage{pgf,pgfarrows,pgfnodes,pgfautomata,pgfheaps,pgfshade,hyperref, amssymb}
\usepackage{amssymb}

\usepackage[english]{babel}
\usepackage[capitalize]{cleveref}

\makeatletter
\def\namedlabel#1#2{\begingroup
    #2%
    \def\@currentlabel{#2}%
    \phantomsection\label{#1}\endgroup
}
\makeatother

\usepackage{mathtools}
\usepackage[colorinlistoftodos]{todonotes}
\usepackage{soul}
\usepackage[utf8]{inputenc} 
\usepackage[T1]{fontenc}
\usepackage[normalem]{ulem}
\usepackage{cancel}

\usepackage{tikz}
\usetikzlibrary{arrows}

\usepackage{tikz}
\usetikzlibrary{arrows}

\usepackage{xcolor}

\usepackage{tabularx}
\usepackage{array, ltablex, multirow}
\usepackage{makecell}
\setcellgapes{6pt}
\makegapedcells

\usepackage{faktor}
\usepackage{xfrac}

\usepackage{enumerate}
\usepackage{comment}

\def\Z{{\mathbb Z}}

\def\N{{\mathbb N}}

\newcommand*\diff{\mathop{}\!\mathrm{d}}

\hypersetup{
    colorlinks,
    linkcolor={blue!30!black},
    citecolor={green!50!black},
    urlcolor={blue!80!black}
} 

\usepackage{mathrsfs} % for \SP and \sT 
\usepackage{dsfont}

\usepackage{float}

\usepackage{graphicx} % Required for inserting images

\usepackage{pgfplots}
\pgfplotsset{width=10cm,compat=1.9}

\newtheorem{theorem}{Theorem}[section]
\newtheorem{proposition}[theorem]{Proposition}

\newtheorem{lemma}[theorem]{Lemma}

\newcounter{thmcounter}

\newcounter{introthmcounter}

\newtheorem{corollary}[theorem]{Corollary}
\theoremstyle{definition}
\newtheorem{definition}[theorem]{Definition}
\newtheorem*{definition*}{Definition}

\newtheorem{question}[theorem]{Question}
\newtheorem*{question*}{Question}
\newcounter{proofcount}
\AtBeginEnvironment{proof}{\stepcounter{proofcount}}% count the proofs

\theoremstyle{remark}
\newtheorem*{remark*}{Remark}

\makeatletter
\def\section{\@startsection{section}{1}%
  \z@{.5\linespacing\@plus.7\linespacing}
{.8\baselineskip}%
  {\normalfont\fontsize{11}{13}\centering\bfseries}%
}

\makeatletter
\def\subsection{\@startsection{subsection}{2}%
  \z@{.4\linespacing\@plus.7\linespacing}
{.6\baselineskip}%
  {\normalfont\centering\bfseries}%
}

\makeatletter                  % reset the claim counter each time proofcount is
\@addtoreset{claim}{proofcount}% increased, effectively this restarts the claims
\makeatother

\theoremstyle{remark}

\usepackage{pgfarrows}

\renewcommand{\d}{~\mathrm{d}}
\def\N{{\mathbb N}}

\newcommand{\xmt}{(X,\mu,T)}

\newcommand{\gen}{\texttt{\textup{gen}}}

\newcommand{\Erdos}{Erd\H{o}s}
\newcommand{\Folner}{F\o{}lner}

\newcommand{\oh}{{\rm o}}

\makeatletter
\renewcommand{\tocsection}[3]{%
  \indentlabel{\@ifnotempty{#2}{\bfseries\ignorespaces#1 #2\quad}}\bfseries#3}
\renewcommand{\tocsubsection}[3]{%
  \indentlabel{\@ifnotempty{#2}{\ignorespaces#1 #2\quad}}#3}
\def\@tocline#1#2#3#4#5#6#7{\relax
  \ifnum #1>\c@tocdepth % then omit
  \else
    \par \addpenalty\@secpenalty\addvspace{#2}%
    \begingroup \hyphenpenalty\@M
    \@ifempty{#4}{%
      \@tempdima\csname r@tocindent\number#1\endcsname\relax
    }{%
      \@tempdima#4\relax
    }%
    \parindent\z@ \leftskip#3\relax \advance\leftskip\@tempdima\relax
    \rightskip\@pnumwidth plus1em \parfillskip-\@pnumwidth
    #5\leavevmode\hskip-\@tempdima{#6}\nobreak
    \leaders\hbox{$\m@th\mkern \@dotsep mu\hbox{.}\mkern \@dotsep mu$}\hfill
    \nobreak
    \hbox to\@pnumwidth{\@tocpagenum{\ifnum#1=1\bfseries\fi#7}}\par% <-- \bfseries for \section page
    \nobreak
    \endgroup
  \fi}
\AtBeginDocument{%
\expandafter\renewcommand\csname r@tocindent0\endcsname{0pt}
}
\def\l@subsection{\@tocline{2}{0pt}{2.5pc}{5pc}{}}
\makeatother

\title{Sharp density conditions for infinite $B+B$ sumsets in abelian groups}
\author{Ioannis Kousek}
\address{Mathematics Institute, University of Warwick, Coventry, UK}
\email{ioannis.kousek@warwick.ac.uk}
\date{}

\begin{document}

\begin{abstract}
Motivated by recent results \cite{charamaras_kousek_mountakis_radic2025BBingroups} on 
infinite sumsets of the form 
$B+B=\{b_1+b_2:b_1,b_2\in B\}$ in 
large subsets of abelian groups, and an old problem of Owings 
\cite[Problem E2494]{Owing_problems} about the partition regularity of $B+B$ in $2$ colours, we 
show the following theorem. Let $(G,+)$ be a countable abelian 
group such that the subgroup $\{g+g\colon g\in G\}$ has finite 
index and the doubling map $D: g\mapsto g+g$ has finite kernel. Let also $\Phi=(\Phi_N)_{N}$ be any \Folner\ sequence in $G$ and $\Phi/2=(D^{-1}(\Phi_N))_{N}$. Then, if 
$A\subset G$ is such that 
$\diff_{\Phi}(A)+\diff_{\Phi/2}(A)>1$, 
there is an infinite set $B\subset G$ and some $t\in G$ for 
which $t+B+B\subset A$.  

We prove that this result implies the main theorem in 
\cite{charamaras_kousek_mountakis_radic2025BBingroups}, and 
construct an example to show the reverse implication does not 
hold. Moreover, we show that our 
main theorem is optimal in a strong sense. Namely, for any 
countable abelian group $(G,+)$ with the aforementioned 
assumptions -- which are necessary -- there exists a \Folner\ sequence $\Phi$ and a set $A\subset G$ so that $\diff_{\Phi}(A)+\diff_{\Phi/2}(A)=1$, but there is no infinite set $B\subset G$ and $t\in G$ for which $t+B+B\subset A$. 
\end{abstract}

\maketitle

\tableofcontents

\section{Introduction}

Motivated by recent breakthroughs 
\cite{ Kra_Moreira_Richter_Robertson:2022, Kra_Moreira_Richter_Robertson:2023,Moreira_Richter_Robertson19} on the general problem of finding 
infinite sumsets in sets with positive density, and specifically by a conjecture of Kra, Moreira, Richter and Robertson \cite[Conjecture 3.7]{Kra_Moreira_Richter_Robertson_problems}, Radi{\'c} and the author \cite{kousek_radic2025BB} proved the first density regularity type of results for unrestricted infinite sumsets. In \cite[Theorems 1.2, 1.3]{kousek_radic2025BB} it was shown, for example, that for $A\subset \N$ with $\underline{\diff}(A)>1/2$, there is an infinite set $B\subset \N$ and some integer $t\geq 0$, such that $t+B+B\subset A$. 

We will refer to \textit{the $B+B$ problem} as the general problem of finding a density scheme -- in a countable abelian group -- for which any set with large enough density contains a sumset of the form $t+B+B$. To define such a density scheme, we use the notion of \textit{\Folner\ sequences}. Recall that given an abelian group $G$, a sequence  
$\Phi=(\Phi_N)_{N\in \N}$ of finite subsets of $G$ is 
a \Folner{} sequence
if for any $g\in G$, 
\begin{equation*}
    \lim_{N\to \infty} 
\frac{|\Phi_N \cap (g+\Phi_N)|}{|\Phi_N|}=1.
\end{equation*}
The \textit{upper density of} $A\subset G$ with
respect to $\Phi$ is defined as   
$\overline{\diff}_{\Phi}(A)  = \limsup_{N\to \infty} \frac{|A\cap \Phi_N|}{|\Phi_N|}$, we speak of the density $\diff_{\Phi}(A)$ when the limit exists, and we say that $A$ has \textit{positive upper Banach density} if $\overline{\diff}_{\Psi}(A)>0$ for some \Folner\ sequence $\Psi$ in $G$. 

Our main result reads as follows. 

\begin{theorem}\label{main theorem}
Let $(G,+)$ be a countable abelian group with $\ell=[G:2G]<\infty$ and $r=|\ker(D)|<\infty$. Let $A\subset G$ and $\Phi$ be any \Folner{} sequence in $G$ such that the densities below exist. Then, the following hold:
\begin{enumerate}
\item  \label{main_theorem_1_2}
If 
$\diff_\Phi(A)+ \diff_{\Phi/2}(A)> 1,$
there exist an infinite set $B\subset G$ and some $t\in G$ such that 
$t+B+B\subset A$. 
\item  \label{main_theorem_1_1} 
\vspace{2mm} 
If 
$\diff_\Phi(A) + \diff_{\Phi/2}(A)> \frac{2\ell-1}{\ell},$
there exists an infinite set $B\subset G$ such that $B+B\subset A$. 
\item   \label{main_theorem_evens}  \vspace{2mm} If 
$ \diff_\Phi(A\cap 2G)+\diff_{\Phi/2}(A\cap 2G) > \frac{1}{\ell},$
there exists an infinite set $B\subset G$ such that $B+B\subset A$.
\end{enumerate}    
\end{theorem}

Starting from a \Folner\ sequence $\Phi$, we may always pass 
to an appropriate subsequence so that the densities in 
Theorem \ref{main theorem} exist. We want to explain why 
this assumption is necessary. As the proof in Section \ref{sec proof} reveals, we want to consider densities of the form 
$$\diff_{(\Phi_{N_k})}(A) + \diff_{(\Phi_{N_k}/2)}(A),$$
along the same subsequence $(N_k)$. This is not a mere 
artifact of the proof, but rather we even have the following extreme behaviour if we consider densities along different subsequences of $\Phi$ and $\Phi/2$. 

\begin{proposition}\label{same subsequence intro}
Let $(G,+)$ be a countable abelian group with 
$\ell=[G:2G]<\infty$ and $r=|\ker(D)|<\infty$. There is a set 
$A\subset G$ and a \Folner\ sequence $\Phi$ such that 
$$\overline{\diff}_{\Phi}(A)+\overline{\diff}_{\Phi/2}(A)=2, $$
but there does not exist an infinite set $B\subset G$ and 
$t\in G$ with $t+B+B\subset A$.        
\end{proposition}

Returning to our original discussion, the results in \cite{kousek_radic2025BB} were restricted to 
taking upper and lower densities with respect to the original 
\Folner\ sequence $N\mapsto [1,N]=\{1,\ldots,N\}$ in $\N$.
The threshold values established therein were proved to be 
optimal, and while it was known that not all \Folner\ 
sequences can be used to obtain analogous results (see, for 
example, \cite[Example 3.6]{Kra_Moreira_Richter_Robertson_problems}), it was not explored
any further whether other density schemes, e.g. density of a set along other \Folner\ sequences, could be used to 
address the $B+B$ problem in $\N$. 

Around the same time, Charamaras and Mountakis 
\cite{Charamaras_Mountakis_2025} proved the analogue of the 
main result in \cite{Kra_Moreira_Richter_Robertson:2023} for 
the widest class of countable abelian groups possible. In 
particular, they showed that for countable abelian groups $G$
with $[G:2G]<\infty$, any set with positive upper Banach density contains an infinite sumset $t+B\oplus B=\{t+b_1+b_2:b_1,b_2\in B,\ \text{and}\ b_1\neq b_2\}$. 
In an effort to extend the results of 
\cite{Charamaras_Mountakis_2025} to cover unrestricted 
sumsets, and to understand more deeply the necessary 
structural properties of \Folner\ sequences that can be used 
as averaging schemes to a density solution for the $B+B$ 
problem, Charamaras, Mountakis, Radi{\'c} and the author 
\cite{charamaras_kousek_mountakis_radic2025BBingroups} 
introduced the following notion. Recall that for a set $C \subset G$ of an abelian group $(G,+)$, we write $2C = D(C)$ and $C/2 = D^{-1}(C)$, where $D\colon G \to G$ is the \emph{doubling map} $g \mapsto 2g$.

\begin{definition}[cf. {\cite[Definition 1.2]{charamaras_kousek_mountakis_radic2025BBingroups}}]\label{quasi-invariant to doubling def}
Let $G$ be a countable abelian group. The {\em the doubling ratio} of a \Folner{} sequence $\Phi=(\Phi_N)_{N\in\N}$ in $G$ is defined as
    \begin{equation} \label{eq alpha_Phi defn}
        \alpha_{\Phi}  = 
        \liminf_{N\to \infty} \frac{| \Phi_N/2 \cap \Phi_N|}{|\Phi_N|}.
    \end{equation}
    If $\alpha_\Phi>0$, we say that the \Folner{} sequence $\Phi$ is 
    {\em quasi-invariant with respect to doubling (q.i.d.) with ratio $\alpha_\Phi$}.
\end{definition}

The authors in 
\cite{charamaras_kousek_mountakis_radic2025BBingroups} showed 
that the q.i.d. assumption can be used to provide density 
threshold values for the $B+B$ problem via the following
theorem. 

\begin{theorem}[{\cite[Theorem 1.3]{charamaras_kousek_mountakis_radic2025BBingroups}}]\label{CKMR}
Let $(G,+)$ be a countable abelian group with $\ell=[G:2G]<\infty$ and $r=|\ker(D)|<\infty$. 
Let $A\subset G$ and $\Phi$ be any \Folner{} sequence in $G$ that is quasi-invariant with respect to doubling with ratio $\alpha_{\Phi}$. Then the following hold:
\begin{enumerate}
\item  \label{main_theorem_1_2'}
If 
$\displaystyle \overline{\diff}_\Phi(A) > 1 - \frac{\ell\alpha_\Phi}{\ell +r},$
there exist an infinite set $B\subset G$ and some $t\in G$ such that 
$t+B+B\subset A$. 
\item  \label{main_theorem_1_1'} 
\vspace{2mm} 
If  
$\displaystyle { \overline{\diff}_\Phi(A) > 1 - \frac{\alpha_\Phi}{\ell+r}},$
there exists an infinite set $B\subset G$ such that $B+B\subset A$. 
\item   \label{main_theorem_evens'}  \vspace{2mm} If 
$\displaystyle\overline{\diff}_\Phi(A\cap 2G) > \frac{1}{\ell} - \frac{\alpha_\Phi}{\ell +r},$
there exists an infinite set $B\subset G$ such that $B+B\subset A$.
\end{enumerate}    
\end{theorem}

To complement this result, it was shown in 
\cite[Proposition 5.4]{charamaras_kousek_mountakis_radic2025BBingroups} 
that if 
a \Folner\ sequence is not q.i.d. there is a set of full 
density along this sequence that contains no sumset as in 
\eqref{main_theorem_1_2'}. In that sense, the q.i.d. 
assumption is necessary for Theorem \ref{CKMR}. This 
proposition (and thus Definition \ref{quasi-invariant to doubling def}) will be used repeatedly in the constructions of Section \ref{optimality sec}.

It is not obvious, but it turns out that Theorem \ref{main theorem} implies Theorem \ref{CKMR} and we prove this in Proposition \ref{implication of main results}. Perhaps even more interesting is the fact that the reverse implication is false. Namely, Theorem \ref{CKMR} cannot be used to deduce Theorem \ref{main theorem} as there exist sets that satisfy, say, condition \eqref{main_theorem_1_2} of Theorem \ref{main theorem} along some \Folner\ sequence, but not condition \eqref{main_theorem_1_2} of Theorem \ref{CKMR} along any \Folner\ sequence. 

\begin{proposition}\label{no reverse implication of main theorems intro}
There is a set 
$A\subset \Z$ such that 
$$\overline{\diff}_{\Phi}(A) \leq 1-\frac{\ell \alpha_{\Phi}}{\ell+r},$$
for every \Folner\ sequence $\Phi$ in $\Z$, but 
$$\diff_{F}(A)+\diff_{F/2}(A)>1, $$
for some \Folner\ sequence $F$ in $\Z$.
\end{proposition}

We also remark that the three statements in Theorem \ref{CKMR} were shown to be equivalent in \cite{charamaras_kousek_mountakis_radic2025BBingroups} and we do the same for the statements of Theorem \ref{main theorem} in Section \ref{section 2}.

The proof of Theorem \ref{main theorem} is dynamical, 
similarly to the proof of Theorem \ref{CKMR}. In fact, the 
main dynamical result behind both theorems is the same. The 
difference lies in the way we implement the correspondence principle, which is used to translate the combinatorial input 
of the statement into an assumption about a corresponding
dynamical system. 

The reason why the stronger Theorem \ref{main theorem} was 
not found in 
\cite{charamaras_kousek_mountakis_radic2025BBingroups} 
has to do with the broader viewpoint of how the problem was 
formulated and thus approached there. In particular, from the perspective of the authors in \cite{charamaras_kousek_mountakis_radic2025BBingroups} ``the main question to address is whether for a given abelian group $G$ one can find a \Folner{} sequence $\Phi$ and a constant $c=c(G,\Phi)>0$ such that any set $A\subset G$ with upper density 
$$\overline{\diff}_\Phi(A) = \limsup_{N \to \infty} \frac{|A \cap \Phi_N|}{|\Phi_N|}>c$$ 
contains an unrestricted sumset of the form $B+B$ for some 
infinite set $B\subset G$.'' 

The reason why Theorem \ref{main theorem} was found now is 
twofold. Firstly, it appeared naturally in considerations of 
questions about the optimality of Theorem \ref{CKMR} 
(optimality in the 
sense of sharpness of the threshold bounds). These questions 
were formulated in \cite[Question 1.6]{charamaras_kousek_mountakis_radic2025BBingroups} and \cite[Question 1.7]{charamaras_kousek_mountakis_radic2025BBingroups}. We present here their analogues for Theorem \ref{main theorem}.
The first question asks whether the result is 
optimal in any group where it holds. 

\begin{question}[{cf. {\cite[Question 1.7]{charamaras_kousek_mountakis_radic2025BBingroups}}}] \label{open_qu_1}
Let $G$ be a countable abelian group with $\ell=[G:2G]<\infty$ and 
$r=|\ker(D)|<\infty$. Can one always find a \Folner{} sequence $\Phi$ in $G$ and a set  
$A\subset G$ with 
$\diff_{\Phi}(A)+\diff_{\Phi/2}(A)=1,$ 
such that $t+B+B \not \subset A$ for any infinite $B\subset G$ and $t\in G$?
\end{question}

The answer to \cite[Question 1.7]{charamaras_kousek_mountakis_radic2025BBingroups} is not known; however, the weaker assumptions in our main result allow us to address its analogue, Question \ref{open_qu_1}, in full generality. 

\begin{theorem}\label{optimality 1 intro}
Let $(G,+)$ be a countable abelian group with 
$\ell=[G:2G]<\infty$ and $r=|\ker(D)|<\infty$. There is a set 
$A\subset G$ such that 
$$\diff_{\Phi}(A)+\diff_{\Phi/2}(A)=1, $$
for some \Folner\ sequence $\Phi$ in $G$, but there does not 
exist an infinite set $B\subset G$ and $t\in G$ with 
$t+B+B\subset A$.    
\end{theorem}

The second question asks for the strongest 
possible notion of optimality for the bounds in our main 
result.

\begin{question}[{cf. {\cite[Question 1.6]{charamaras_kousek_mountakis_radic2025BBingroups}}}] \label{open_qu_0}
Let $G$ be a countable abelian group with $\ell=[G:2G]<\infty$ and 
$r=|\ker(D)|<\infty$, and let $\Phi$ be a \Folner{} sequence in $G$. Does there exist a set $A\subset G$ with 
$\diff_{\Psi}(A)+\diff_{\Psi/2}(A)=1$, along some subsequence $\Psi$ of $\Phi$, and such that $t+B+B \not \subset A$ for any infinite $B\subset G$ and $t\in G$?
\end{question}

Answering Question \ref{open_qu_0} is probably very 
difficult in full generality (unless the answer is negative 
and we are missing an easy 
counterexample), but we can address it in the 
integers along sequences of intervals. Indeed, this relates 
to the second reason why Theorem \ref{main theorem} was 
found. An old Problem of C.J. Owings \cite[Problem E2494]{Owing_problems} asks whether, for any colouring (partition) of $\N$ into 
two sets, 
one of the sets must contain $B+B$ for some infinite 
$B\subset \N$. This problem was studied by various authors 
(see, for example, \cite{Owings1, Owings2}) and is rather interesting because
Hindman \cite{Hindman79_b} showed the analogue for 
three-colourings to have a negative answer. Very recently it 
was resolved in the nice work of Huang, Lian, 
Shao, Xiao, Xu and Zhang \cite{owings_solution}.  

In the 
author's effort to address this problem with some type of a 
density argument (which was motivated by the joint work with 
Radi{\'c} \cite{kousek_radic2025BB}), it was observed that if 
Theorem \ref{main theorem} was not optimal, Owings' 
question would have a positive answer. The details are 
explained in Section \ref{Owing}. 

Here, now, is the special case of Question 
\ref{open_qu_0} we can address. 

\begin{proposition}\label{optimality in Z intro}
Let $\Phi$ be any \Folner\ sequence of intervals in $\Z$. 
Then, there 
exists a set $A\subset \Z$ such that 
$\diff_{\Psi}(A)+\diff_{\Psi/2}(A)=1$, 
where $\Psi$ is a subsequence of $\Phi$ along which the 
densities exist, for which $t+B+B \not \subset A$, for any 
infinite set $B$ 
and any $t\in \Z$.
\end{proposition}

Finally, we want to briefly explain the assumptions that 
$[G:2G]<\infty$ and $r=|\ker(D)|<\infty$ in all our results. 
In 
\cite[Section 5.1]{charamaras_kousek_mountakis_radic2025BBingroups} it was 
shown that there exists an abelian group with infinite 
$\ker(D)$, and a set of full density along a \Folner{} 
sequence that is q.i.d., which contains no infinite sumsets 
$B+B$. It was also proved in \cite[Section 5.3]{charamaras_kousek_mountakis_radic2025BBingroups} that when $2G$ is infinite and $[G:2G]=\infty$, Theorem \ref{CKMR} 
fails even with density $1$. The second assumption was also 
imposed by the results in \cite{ackelsberg2024counterexamples}. Since these assumptions 
were already known to be necessary for Theorem \ref{CKMR}, they have to be assumed for Theorem \ref{main theorem} to hold as well, in view of the relation between the two results (Proposition \ref{implication of main results}).

\vspace{2mm} 
\noindent
\textbf{Acknowledgements.} 
The author would like to thank Joel Moreira and Trist{\'a}n 
Radi{\'c} for suggestions that improved the exposition.

\section{Theorem \ref{main theorem} implies Theorem \ref{CKMR}}\label{section 2}

As was mentioned before, $G$ always denotes a countable 
abelian group with $\ell=[G: 2G]< \infty$. 
We also fix $g_1, \ldots, g_{\ell}\in G$ so that 
$G=\bigsqcup_{i=1} ^{\ell} 2G+g_i$. We omit these assumptions 
from the statements of this section. 

The main purpose of this section is to prove that our main 
result implies the main result of 
\cite{charamaras_kousek_mountakis_radic2025BBingroups}, Theorem \ref{CKMR}. Since 
the conclusion of both theorems is the same, this amounts to 
showing that the assumptions of Theorem \ref{main theorem} are weaker 
than those of Theorem \ref{CKMR}. We begin by showing
in Proposition \ref{prop equiv formulation} 
that the three statements of Theorem \ref{main theorem} are
equivalent, in analogy to 
\cite{charamaras_kousek_mountakis_radic2025BBingroups}. This 
is also useful when showing that the bounds are optimal, 
because we only need to do this for one of the bounds. 

We recall three auxiliary results that were established in \cite{charamaras_kousek_mountakis_radic2025BBingroups}. 

\begin{lemma} \label{lemma reductions no shift}
        Let $A \subset G$. The following are equivalent:
    \begin{enumerate}
        \item \label{lemma reductions no shift_1} $A\supset B+B$ for some 
        infinite set $B\subset G$.
        \item \label{lemma reductions no shift_2} $A \cap 2G \supset B+B$ for some infinite set $B\subset G$.
        \item \label{lemma reductions no shift_3} $(A \cap 2G) \cup(G \backslash 2G) \supset B+B$ for some infinite set $B\subset G$.
    \end{enumerate}
\end{lemma}

\begin{lemma} \label{lemma reductions with shift}
Let $A \subset G$. The following are equivalent:
\begin{enumerate}
    \item \label{lemma reductions with shift_1} $A\supset t+B+B$ for some $t \in G$ and infinite $B \subset G$.
    \item \label{lemma reductions with shift_2} $A \supset B+B + g_i$ for some $i \in \{1, \ldots, \ell\}$ and infinite 
    $B \subset G$.
    \item \label{lemma reductions with shift_3} $(A - g_i) \cap 2G  \supset B+B$ for some $i \in \{1, \ldots, \ell\}$ and infinite 
    $B \subset G$.
\end{enumerate}
\end{lemma}

\begin{lemma}\label{lemma aux folner 2}
For any \Folner{} sequence $\Phi=(\Phi_N)_{N\in \N}$ in $G$ we have that 
$\frac{|\Phi_N/2|}{|\Phi_N \cap 2G|}= r + \oh_{N\to \infty} (1) $ 
and $\frac{|\Phi_N/2|}{|\Phi_N|}= \frac{r}{\ell} + \oh_{N\to \infty} (1)$.    
\end{lemma}

\begin{proposition} \label{prop equiv formulation}
    Let $(G,+)$ be a countable abelian group with $\ell=[G:2G]<\infty$. Let $A\subset G$ and $\Phi$ be any \Folner{} sequence in $G$. Then the following statements are equivalent:
\begin{enumerate}
\item  \label{prop equiv formulation_1}
If 
$\diff_\Phi(A)+ \diff_{\Phi/2}(A)> 1,$
there exist an infinite set $B\subset G$ and some $t\in G$ such that 
$t+B+B\subset A$. 
\item  \label{prop equiv formulation_2} 
\vspace{2mm} 
If 
$\diff_\Phi(A) + \diff_{\Phi/2}(A)> \frac{2\ell-1}{\ell},$
there exists an infinite set $B\subset G$ such that $B+B\subset A$. 
\item   \label{prop equiv formulation_3}  \vspace{2mm} If 
$\diff_\Phi(A\cap 2G)+\diff_{\Phi/2}(A\cap 2G) > \frac{1}{\ell},$
there exists an infinite set $B\subset G$ such that $B+B\subset A$.
\end{enumerate} 
\end{proposition}

\begin{proof}
The proof is very similar to that of \cite[Proposition 2.3]{charamaras_kousek_mountakis_radic2025BBingroups}, and as an indicative example we prove the equivalence between \eqref{prop equiv formulation_1} and \eqref{prop equiv formulation_3}.

\eqref{prop equiv formulation_1}$\implies$\eqref{prop equiv formulation_3}: If $\diff_{ \Phi}(A \cap 2G )+\diff_{ \Phi/2}(A \cap 2G ) > \frac{1}{\ell}$, we let $\Tilde{A} = \bigsqcup_{i=1}^{\ell} (A \cap 2G) + g_i$, and the invariance of $\Phi$ and $\Phi/2$ gives that $\diff_{ \Phi}(\Tilde{A} )+\diff_{ \Phi/2}(\Tilde{A} ) > 1$. By \eqref{prop equiv formulation_1}, $\Tilde{A}$ contains a sumset of the form $t+B+B$ for some infinite $ B \subset G$ and $t \in G$.
    By Lemma \ref{lemma reductions with shift} we can reduce to the case where $t= g_i$ for some $i \in \{1, \ldots, \ell\}$. Using 
    Lemma \ref{lemma reductions no shift}, we can also assume that $B+B \subset 2G$. Therefore, we have that $B+B+ g_i \subset (A \cap 2G) + g_i$ which concludes the proof.

    \eqref{prop equiv formulation_3}$\implies$\eqref{prop equiv formulation_1}: Suppose $\diff_{ \Phi}(A )+\diff_{ \Phi/2}(A) > 1$. By sub-additivity of the density, 
$$\diff_{ \Phi}(A)+\diff_{ \Phi/2}(A) \leq \sum_{i=1}^{\ell} \diff_{ \Phi}(A \cap (2G + g_i) )+\diff_{ \Phi/2}(A \cap (2G + g_i) ), $$ so there is $i \in \{1, \ldots, \ell\}$ for which     
    $\diff_{ \Phi}(A \cap (2G + g_i) )+\diff_{ \Phi/2}(A \cap (2G + g_i) )  >\frac{1}{\ell}$. 
    By translation invariance
    of the density and \eqref{prop equiv formulation_3}, it follows that $A$ contains $B+B+g_i$ for some $ B \subset G$ infinite.  
\end{proof}

We now prove the main result of this section. 

\begin{proposition}\label{implication of main results}
Theorem \ref{main theorem} implies Theorem \ref{CKMR}.    
\end{proposition}

\begin{proof}
    In view of Proposition \ref{prop equiv formulation} and \cite[Proposition 2.3]{charamaras_kousek_mountakis_radic2025BBingroups}, it suffices to show that \eqref{main_theorem_1_2} of Theorem \ref{main theorem} implies \eqref{main_theorem_1_2'} of Theorem \ref{CKMR}. As explained in the beginning of this section, we only have to prove that that the assumptions of the latter imply the assumptions of the former. 

Let $\Phi=(\Phi_N)$ be a \Folner\ sequence in $G$ with doubling ratio 
$$0<\alpha_{\Phi}=\liminf_{N\to \infty} \frac{| \Phi_N/2 \cap \Phi_N|}{|\Phi_N|}$$
and, by passing to a subsequence of $\Phi$ if necessary, 
suppose that $A\subset G$ with $\diff_{\Phi}(A)>1-\frac{\ell \alpha_{\Phi}}{\ell+r}$ and that $\diff_{\Phi/2}(A)$ exists as well. We will next verify that $\diff_{\Phi}(A)+\diff_{\Phi/2}(A)>1$.
Observe that 
$$ \left| A \cap \Phi_N \right| =  \left| A \cap \Phi_N \cap \Phi_N/2 \right| + \left| A \cap \left(\Phi_N \setminus \Phi_N/2 \right) \right| \leq \left| A \cap \Phi_N \cap \Phi_N/2 \right| + \left| \Phi_N \right| - \left| \Phi_N \cap \Phi_N/2 \right|.$$
Using the density assumption and the doubling ratio property, it follows that 
\begin{equation}\label{eq: 1}
\limsup_{N\to \infty} \frac{\left| A \cap \Phi_N \cap \Phi_N/2 \right|}{\left| \Phi_N \right|} > \alpha_{\Phi}\left(1-\frac{\ell}{\ell+r}\right)= \left(\frac{r \alpha_{\Phi}}{\ell+r}\right).     
\end{equation}
Using Lemma \ref{lemma aux folner 2}, we see that 
$$\diff_{\Phi/2}(A)= \lim_{N\to \infty} \frac{\left| A\cap \Phi_N/2 \right|}{ \left| \Phi_N/2 \right|}=\lim_{N\to \infty} \frac{\ell \left| A\cap \Phi_N/2 \right|}{r \left| \Phi_N \right|}.$$
Hence, in view of \eqref{eq: 1}, we have that
\begin{align*}
& \diff_{\Phi}(A)+\diff_{\Phi/2}(A)> 1-\frac{\ell \alpha_{\Phi}}{\ell+r} + \lim_{N\to \infty} \frac{\ell \left| A\cap \Phi_N/2 \right|}{r \left| \Phi_N \right|} \\ & \geq  1-\frac{\ell \alpha_{\Phi}}{\ell+r} + \limsup_{N\to \infty} \frac{\ell \left| A\cap \Phi_N \cap \Phi_N/2 \right|}{r \left| \Phi_N \right|} > 1-\frac{\ell \alpha_{\Phi}}{\ell+r} + \left(\frac{\ell \alpha_{\Phi}}{\ell+r}\right)=1.
\end{align*}
\end{proof}

\section{Proof of the main result}\label{sec proof}

\subsection{Ergodic background}

We first explain the essential ergodic terminology and
notation necessary to understand the statement of our main 
dynamical result which is borrowed from 
\cite{charamaras_kousek_mountakis_radic2025BBingroups}. 

By $G$ we denote a countable abelian group. Given a compact 
metric space $X=(X,d_X)$, 
a \textit{continuous action} $T=(T_g)_{g\in G}$ of $G$ on X
is a collection of continuous functions $T_g:X\to X$ such 
that 
for any $g_1, g_2 \in G$, $T_{g_1} \circ T_{g_2} = T_{g_1+g_2}$.
Given such an action, we call the pair $(X,T)$ a \textit{topological $G$-system.}

A Borel probability measure $\mu$ on $X$ is called 
\textit{$T$-invariant}, if it is invariant under $T_g$ for 
all $g\in G$. 
Then, the action $T$ on the space 
$(X,\mu)$ is called a \textit{measure-preserving $G$-action} 
and $\xmt$ is called a \textit{measure-preserving $G$-system} or \textit{G-system} for short. 

A $G$-system $\xmt$ is \textit{ergodic} 
if $ T_g ^{-1} A =A \text{ for all } g\in G$ implies that $\mu(A) \in \{0,1\},$ 
for any measurable set $A$.

Given a $G$-system $\xmt$ and a \Folner{} sequence $\Phi$, a point $a\in X$ is called 
\textit{generic with respect to $\mu$ along $\Phi$}, which we write as $a\in \gen(\mu,\Phi)$, if 
$$ \lim_{N\to \infty} \frac{1}{| \Phi_N|} \sum_{g\in \Phi_N}
\delta_{T_g a}=\mu,$$
where the limit 
is in the weak$^\ast$ topology and $\delta_x$ denotes the 
Dirac mass at $x\in X$.

We also need the dynamical counterpart of infinite 
sumsets, which are called \textit{Erd\H{o}s progressions}, 
and were originally introduced for $\Z$-actions in 
\cite{Kra_Moreira_Richter_Robertson:2023}.
Given a topological $G$-system $(X,T)$, a triple 
$(x_0,x_1,x_2) \in X^3$ is called a 
($3$-term) Erd\H{o}s progression if there exists 
an infinite sequence $(g_n)_{n\in \N}$ in $G$ (that is, the set 
$\{g_n : n\in \N\}$ is infinite)
such that $(T_{g_n} \times T_{g_n})(x_0,x_1) \xrightarrow{} (x_1,x_2)$ 
as $n\to \infty$.

The relation between these dynamical progressions and sumsets 
is captured by the following lemma. 

\begin{lemma}\cite[Lemma $3.4$]{Charamaras_Mountakis_2025} \label{EP and B+B} 
Let $(X,T)$ be a topological $G$-system and let $E,F\subset X$ be open. Assume 
there exists an Erd\H{o}s progression $(x_0,x_1,x_2) \in X^3$ with $x_1\in E$ 
and $x_2\in F$. Then, there exists an infinite sequence $B=(b_n)_{n\in \N} 
\subset \{g\in G: T_g(x_0) \in E\}$ such that $B \oplus B =
\{b_n + b_m \colon n,m \in \N, n \neq m\} \subset \{g\in G: 
T_g(x_0)\in F\}$.
\end{lemma}

Finally, the main dynamical theorem that implies Theorem 
\ref{main theorem} was proven in \cite{charamaras_kousek_mountakis_radic2025BBingroups} and is this. 

\begin{theorem}{{\cite[Theorem 2.5]{charamaras_kousek_mountakis_radic2025BBingroups}}} \label{analogue of 2.1 KR}
Let $(X,\mu,T)$ be an ergodic $G$-system, $a\in \gen(\mu,\Phi)$ for some \Folner{} sequence $\Phi$ in $G$,  
and $E_1,E_2 \subset X$ be open sets satisfying
\begin{equation}\label{1_1}
\ell \mu(E_2)+\mu(E_1) > \ell.
\end{equation}
Then, there exists an \Erdos{} progression $(a,x_1,x_2)$ such that $(x_1,x_2) \in E_1 \times E_2$.
\end{theorem}

\subsection{A correspondence principle}
\label{section_dynamical_st}

As we explained in the introduction, the difference between 
the proofs of Theorems \ref{main theorem} and \ref{CKMR} lies 
at the level of the correspondence principle. More precisely, 
starting from a set $A\subset G$ and a \Folner\ sequence 
$\Phi$, the correspondence principle associates to $(A,\Phi)$ 
an ergodic system which, in some sense, encodes the 
statistical behavior of $A$ along $\Phi$. The improvement of 
this approach is that allowing for a more flexible averaging 
scheme, where we 
consider the density of a set along two distinct, but 
related, \Folner\ sequences, allows us to obtain a better 
lower bound for the measures of some important sets 
in the ergodic system associated to $A$. 

Again, let $(G,+)$ be an abelian group with $\ell=[G: 2G] < \infty$ and $r=|\ker(D)|<\infty$. By $\Sigma$ we denote the space $\{0,1\}^{G}$, endowed 
with the product topology, which becomes a compact metrizable 
space. 
We also consider the shift action $S=(S_g)_{g\in G}$ defined by
$S_g\colon \Sigma \to \Sigma$, $S_g(x(h))=x(h+g)$, for any 
$h,g\in G$, $x=(x(g))_{g\in G} \in \Sigma$, and note that 
$S$ is an action of $G$ on $X$ by homeomorphisms.
We further denote the $G$-action $g\mapsto S_{2g}$ by $S^2$.

Via the following variant of Furstenberg's correspondence 
principle, (originally introduced in Furstenberg's ergodic 
proof of Szemer\'edi's theorem \cite{Furstenberg77}), we can 
reduce 
Theorem \ref{main theorem} to the dynamical statement of 
Theorem \ref{CKMR}.

\begin{lemma}\label{correspondence}
Let $A\subset G$ and $\Phi$ be any \Folner{} sequence in $G$ 
(and pass to a subsequence if necessary so that all densities 
exist). 
There exist an ergodic $G$-system $(\Sigma 
\times \Sigma, \mu, S^2 \times S)$, an open set $E\subset \Sigma$, a point $a\in \Sigma$ and a \Folner{} sequence 
$\Phi'$, such that $(a,a)\in \gen(\mu,\Phi')$, $A=\{g\in G\colon S_g a \in E\}$ and 
\begin{equation}\label{dbound_no_shift}
    \ell\mu(\Sigma\times E) + \mu(E\times\Sigma) 
    \geq \ell \cdot (\diff_{\Phi/2}(A) + \diff_{\Phi}(A) ) - \ell +1.
\end{equation}
\end{lemma}

\begin{proof}

By definition, 
$$\diff_{\Phi/2}(A) = \lim_{N\to\infty}\frac{|A\cap\Phi_{N}/2|}{|\Phi_{N}/2|}.$$
We associate to the set $A$ a point $a\in\Sigma=\{0,1\}^G$ via 
$$a(g) = \begin{cases}
    1, & \text{if}~g\in A, \\
    0, & \text{otherwise}.
\end{cases}$$
\ \\
Letting $E=\{x\in\Sigma\colon x(e_G)=1\}$ -- which, we note, 
is a clopen set -- observe that, by 
construction, $A=\{g\in G\colon S_ga\in E\}$. We 
can define the sequence of Borel probability measures $(\mu_N)_{N\in \N}$ on 
$\Sigma\times\Sigma$ given by 
$$\mu_N = \frac{1}{|\Phi_{N}/2|}\sum_{g\in \Phi_{N}/2}\delta_{(S_{2g}\times S_g)(a,a)}.$$
If $\mu'$ is a weak* accumulation point of $(\mu_N)_{N\in \N}$, then $\mu'$ is an $(S^2 \times S)$-invariant measure. Now, since for each $N\in \N$, we have
$$\mu_N(\Sigma\times E) = \frac{1}{|\Phi_{N}/2|}\sum_{g\in \Phi_{N}/2}\delta_{S_ga}(E) = \frac{| A\cap \Phi_N/2|}{|\Phi_{N}/2|},$$
sending $N\to\infty$ yields
\begin{equation}\label{dbound_eq1}
    \mu'(\Sigma\times E) = \diff_{\Phi/2}(A).
\end{equation}
Now, \cite[Lemma A.4]{charamaras_kousek_mountakis_radic2025BBingroups} gives that
$$ \lim_{N\to \infty} \frac{|\{g\in \Phi_N/2 : g+ \ker(D) \subset \Phi_N/2\} |}
    {|\Phi_N/2|}=1.$$
Moreover, if we consider the \Folner\ sequence $N\mapsto F_{N}=\bigcup_{g\in \Phi_{N}/2}
g+\ker(D) \supset \Phi_{N}/2$ (where $(F_N)$ is \Folner\ since $|\ker(D)|=r<\infty)$, it is easy to see that $\sum_{g\in F_{N}} \delta _{S_{2g}a} (E)= r 
\sum_{g\in 2 (\Phi_{N}/2)} \delta _{S_{g}a} (E)$. Therefore, it follows that
\begin{equation*}
    \left| \frac{1}{|\Phi_{N}/2|} \sum_{g\in \Phi_{N}/2} \delta _{S_{2g}a} (E) 
    - \frac{r}{|\Phi_{N}/2|} \sum_{g\in 2 (\Phi_{N}/2)} \delta _{S_{g}a} (E)
    \right| \leq 
    \frac{|F_{N} \setminus \Phi_{N}/2|}{|\Phi_{N}/2|}\xrightarrow[]{N\to \infty} 0.
\end{equation*}
We thus see that
\begin{align}\label{eq:3}
    \mu_N(E\times\Sigma) = \frac{1}{|\Phi_{N}/2|}
    \sum_{g\in\Phi_{N}/2}\delta_{S_{2g}a}(E)
    & = \frac{r}{|\Phi_{N}/2|}\sum_{g\in 2(\Phi_{N}/2)}\delta_{S_ga}(E) + \oh_{N\to\infty}(1) \notag \\
    & \geq \frac{\ell}{|\Phi_{N}|}|A\cap 2(\Phi_{N}/2)|+ \oh_{N\to\infty}(1),
\end{align}
by Lemma \ref{lemma aux folner 2}. Since $2(\Phi_N/2)=\Phi_N \cap 2G$, we have 
$$|A\cap 2(\Phi_{N}/2)|=|A\cap \Phi_N|-|A\cap (\Phi_N\setminus 2(\Phi_N/2))| \geq |A\cap \Phi_N|-|\Phi_N|+| \Phi_N\cap 2G|.$$
Hence, using \cite[Lemma 5.4]{Charamaras_Mountakis_2025}, it 
follows by \eqref{eq:3} as $N\to \infty$ that
\begin{equation}\label{dbound_eq2}
    \mu'(E\times\Sigma) 
    \geq \ell \bigg(\diff_\Phi(A) - 1 + 
    \frac{1}{\ell}\bigg)
    = \ell \diff_\Phi(A) - \ell + 1.
\end{equation}
Combining \eqref{dbound_eq1} and \eqref{dbound_eq2} 
we obtain \eqref{dbound_no_shift} for $\mu'$. 
Although $\mu'$ is not necessarily ergodic, the argument that allows us to find an ergodic measure $\mu$ satisfying \eqref{dbound_no_shift} is standard (see, e.g., the proof of \cite[Lemma 2.4]{charamaras_kousek_mountakis_radic2025BBingroups}). 
\end{proof}

\subsection{The proof of Theorem \ref{main theorem}}

We can now use the correspondence principle of Lemma 
\ref{correspondence} and Lemma \ref{EP and B+B} in order to
deduce \eqref{main_theorem_1_1} of Theorem \ref{main theorem}, which, in view of Proposition \ref{prop equiv formulation} shows the full theorem. 

\begin{proof}[Proof of Theorem \ref{main theorem}]
Let $G, \ell, r, \Phi$ be as in the assumptions of Theorem 
\ref{main theorem}. Let also $A\subset G$ satisfy  
$\diff_\Phi(A) + \diff_{\Phi/2}(A) > \frac{2\ell -1}{\ell}.$
Using Lemma \ref{correspondence} we can find an ergodic $G$-system $(\Sigma 
\times \Sigma, \mu, S^2 \times S)$, an open set $E\subset \Sigma$, a point $a\in \Sigma$ and a \Folner{} sequence 
$\Phi'$ in $G$, such that $(a,a)\in \gen(\mu,\Phi')$, $A=\{g\in G\colon S_g a \in E\}$ and 
\begin{equation*}
     \ell\mu(\Sigma\times E) + \mu(E\times\Sigma) 
    \geq \ell \cdot (\diff_{\Phi/2}(A) + \diff_{\Phi}(A) ) - \ell +1.
\end{equation*}
Because of the assumptions on $A$, we get
$$\ell\mu(\Sigma\times E) + \mu(E\times\Sigma) >\ell.$$ 
So, by Theorem \ref{analogue of 2.1 KR} for the system $(\Sigma \times \Sigma, \mu, S^2 \times S)$, the open sets $\Sigma \times E, E \times \Sigma$ and the generic point $(a,a)\in 
\gen(\mu, \Phi ')$ we find an \Erdos{} progression $((a,a), (x_1,x_2), (y_1,y_2)) \in \Sigma ^6 $ with 
$$(x_1, x_2, y_1, y_2)\in E \times \Sigma \times \Sigma \times E.$$
Applying Lemma \ref{EP and B+B} for the open sets $E\times \Sigma$ and 
$\Sigma \times E$, we obtain an infinite sequence $B=(b_n)_{n\in \N}$
so that 
\begin{equation}\label{eq_useful_2B}
    B \subset \{g\in G: S_{2g}\times S_{g} (a,a) \in E\times \Sigma\}
    =\{ g\in G: S_{2g} a \in E \}
\end{equation}
and 
\begin{equation}\label{eq_useful_B+B}
    B \oplus B \subset \{g\in G: S_{2g}\times S_{g} (a,a) \in \Sigma \times E\}
    =\{g\in G: S_{g} a \in E\}=A.
\end{equation}
We write $2B=\{2b: b\in B\}$.
From \eqref{eq_useful_2B} it follows that 
$2B\subset A$, and combining this with \eqref{eq_useful_B+B}, noting that $B+B=
(B \oplus B) \cup 2B$, we conclude that $B+B\subset A$.
\end{proof}

\section{A slightly stronger result}

With not much more effort, we can actually prove something 
slightly more general than Theorem \ref{main theorem}. 

\begin{theorem}\label{main theorem strong}
Let $(G,+)$ be a countable abelian group with $\ell=[G:2G]<\infty$ and $r=|\ker(D)|<\infty$. Let $A,A'\subset G$ and $\Phi$ be any \Folner{} sequence in $G$. Then the following hold:
\begin{enumerate}
\item  \label{main_theorem_1_2_strong}
If 
$\diff_\Phi(A')+ \diff_{\Phi/2}(A)> 1,$
there exist an infinite set $B\subset G$ and some $t\in G$ such that 
$t+B\oplus B\subset A$ and $t+2B\subset A'$. 
\item  \label{main_theorem_1_1_strong} 
\vspace{2mm} 
If 
$\diff_\Phi(A') + \diff_{\Phi/2}(A)> \frac{2\ell-1}{\ell},$
there exists an infinite set $B\subset G$ such that $B\oplus B\subset A$ and $2B\subset A'$. 
\item   \label{main_theorem_evens_strong}  \vspace{2mm} If 
$ \diff_\Phi(A'\cap 2G)+\diff_{\Phi/2}(A\cap 2G) > \frac{1}{\ell},$
there exists an infinite set $B\subset G$ such that $B\oplus B\subset A$ and $2B\subset A'$.
\end{enumerate}    
\end{theorem}   

The proof of the equivalences between the statements in 
Theorem \ref{main theorem strong} is similar to the proof of Proposition \ref{prop equiv formulation} and so we omit it. In 
order to actually prove Theorem \ref{main theorem strong} we 
use the same dynamical result that we used for Theorem 
\ref{main theorem}, that is Theorem \ref{analogue of 2.1 KR}, 
and a more general form of the correspondence principle to 
capture the statistics of both sets. A similar form of the 
correspondence was used in \cite[Lemmas 2.8, 2.9]{kousek2025asymmetric}.

\begin{lemma}\label{correspondence_2}
Let $A,A'\subset G$ and $\Phi$ be any \Folner{} sequence in 
$G$ (and pass to a subsequence if necessary so that all 
densities exist). There exist an ergodic $G$-system $(\Sigma 
\times \Sigma, \mu, S^2 \times S)$, an open set $E\subset \Sigma$, and points $a,a'\in \Sigma$ and a \Folner{} sequence 
$\Phi'$, such that $(a,a')\in \gen(\mu,\Phi')$, $A=\{g\in G\colon S_g a \in E\}$ and $A'=\{g\in G\colon S_{g} a' \in E\}$ with
\begin{equation}\label{dbound_no_shift}
    \ell\mu(\Sigma\times E) + \mu(E\times\Sigma) 
    \geq \ell \cdot (\diff_{\Phi/2}(A) + \diff_{\Phi}(A') ) - \ell +1.
\end{equation}
\end{lemma}

\begin{proof}
The only difference with the proof of Lemma \ref{correspondence} is that we now introduce two points $a,a'\in \Sigma$ which correspond to the indicators of $A$ and $A'$, respectively, and then consider the measures 
$(\mu_N)_{N\in \N}$ on 
$\Sigma\times\Sigma$ given by 
$$\mu_N = \frac{1}{|\Phi_{N}/2|}\sum_{g\in \Phi_{N}/2}\delta_{(S_{2g}\times S_g)(a',a)}.$$
The rest of the proof proceeds analogously. 
\end{proof}

The question of whether Theorem \ref{main theorem strong} is 
optimal is easily seen to be a trivial one, which makes it 
less interesting than the more natural results, Theorem \ref{CKMR} and Theorem 
\ref{main theorem}. Indeed, taking any \Folner\ sequence $\Phi$ and any set $A\subset G$ with full 
density along $\Phi$ and $A'=\emptyset$, we see at once that 
the threshold values in \eqref{main_theorem_1_2_strong} and \eqref{main_theorem_evens_strong} of 
the theorem are obtained, but the 
conclusion fails. 

\section{Optimality and relations between main theorems}\label{optimality sec}

We next show that Theorem \ref{CKMR} does not imply Theorem 
\ref{main theorem}; that is, the density assumption of Theorem
\ref{CKMR} cannot be used to deduce the density assumption of 
Theorem \ref{main theorem}. This is already the case in the 
integers, which was rather unexpected. However, the 
construction seems to be adaptable in all groups satisfying 
the standing assumptions, although we do not pursue this in 
detail. This was mentioned as Proposition \ref{no reverse implication of main theorems intro} and we repeat the statement here for convenience.    

\begin{proposition}\label{no reverse implication of main theorems}
There is a set 
$A\subset \Z$ such that 
$$\overline{\diff}_{\Phi}(A) \leq 1-\frac{\ell \alpha_{\Phi}}{\ell+r},$$
for every \Folner\ sequence $\Phi$ in $\Z$, but 
$$\diff_{F}(A)+\diff_{F/2}(A)>1, $$
for some \Folner\ sequence $F$ in $\Z$.
\end{proposition}

\begin{proof}
We let $F=(F_N)_{N\in \N}$ be a \Folner\ sequence that is 
not q.i.d. and such that $F_1,F_1/2,F_1/4,F_2,F_2/2,F_2/4,\ldots$ are pairwise disjoint. For example, let $F_N=[4^{2N},2\cdot 4^{2N})\cap \N$, for 
$N\in \N$. In particular, $F_N/2=[2\cdot 4^{2N-1},4^{2N})\cap \N$ and $F_N/4=[4^{2N-1},2\cdot 4^{2N-1})\cap \N$. We 
then consider a set $A\subset \N$ defined via 
$$A= \left(\bigcup_{N=1}^{\infty} F_N\right) \bigcup \left( \bigcup_{N=1}^{\infty} F_N/2 \cap k\N \right),$$
with $k\in \N$, $k\geq 3$. Essentially, this could 
be generalized by replacing $k\N$ with a syndetic set of 
sufficiently small density. 

First, we observe that 
$$\diff_{F}(A)+\diff_{F/2}(A)=1+\frac{1}{k}>1.$$ 
We remark that, in view of Theorem \ref{main theorem}, this 
implies the existence of some infinite set $B\subset \N$ and 
$t\in \N_0$ such that $t+B+B \subset A$; this is, however, 
not a necessary part of this proof. 

We proceed to prove that for every \Folner\ sequence $\Phi$ 
in 
$\Z$, it holds that 
\begin{equation}\label{eq:4} 
\overline{\diff}_{\Phi}(A) \leq 1-\frac{\ell \alpha_{\Phi}}{\ell+r}.    
\end{equation}
By the definition of $A$, we only have to consider \Folner\ 
sequences that only intersect non-trivially with $F$ and 
$F/2$. For, if $\Phi_N= G_N \cup B_N$, where $G_N \subset \bigcup_{N=1}^{\infty} F_{N}\cup F_{N}/2$ and $B_N\cap \left(\bigcup_{N=1}^{\infty} F_{N}\cup F_{N}/2\right)=\emptyset$, it holds that
$$\overline{\diff}_{\Phi}(A)=\limsup_{N\to \infty} \frac{|A\cap \Phi_N|}{|\Phi_N|}= \limsup_{N\to \infty} \frac{|A\cap G_N|}{|G_N|+|B_N|} \leq \limsup_{N\to \infty} \frac{|A\cap G_N|}{|G_N|}.$$
% we also need to check that if 
% $$\limsup_{N\to \infty} \frac{|A\cap G_N|}{|G_N|}\leq 1-\frac{2\alpha_G}{3}, then $$\limsup_{N\to \infty} \frac{|A\cap \Phi_N|}{|\Phi_N|} \leq  1-\frac{2\alpha_{\Phi}}{3}$$, but this works by the above and writing 
%$$\limsup_{N\to \infty} \frac{|A\cap G_N|}{|G_N|+|B_N|}=
%\limsup_{N\to \infty} \frac{|A\cap G_N|}{|G_N|} \cdot \frac{|G_N|}{|G_N|+|B_N|}$$
%and then also comparing the $\alpha$s by opening up the definition. 
To easy the presentation, we first address two simple cases.
If $\Phi=F$ or $\Phi=F/2$, we have that $\alpha_{F}=0$, and 
so \eqref{eq:4} holds by default. Next, we consider the case that $(\Phi_N)_{N\in \N}=(F_N\cup F_N/2)_{N\in \N}$. Then we have 
that 
$$\alpha_{\Phi}= \lim_{N\to \infty} 
\frac{|F_N/2|}{|F_N|+|F_N/2|}=\frac{1}{3}.$$
In the integers we have that $\ell=2$ and $r=1$, so this means 
that the right hand side on \eqref{eq:4} equals $7/9$. At the 
same time 
$$\overline{\diff}_{\Phi}(A)=\lim_{N\to \infty} \frac{|F_N|+|F_N/2|/k}{|F_N|+|F_N/2|} = \frac{2+1/k}{3}=\frac{2k+1}{3k} \leq \frac{7}{9}, $$
for $k\geq 3$.
Similar calculations show \eqref{eq:4} in the case that 
$\Phi_N=(G_N \cup H_N)_{N\in \N}$, where $G_N\subset F_N$ and $H_N\subset F_N/2$.

Finally, we consider the general case where $\Phi_N=(G_N \cup H_N)_{N\in \N}$, with $G_N=G_{N_{k_1}}\cup \cdots \cup G_{N_{k_N}}$ for $G_{N_{k_i}}\subset F_{n(N,k_i)}$, and $H_N=H_{N_{m_1}}\cup \cdots \cup H_{N_{m_N}}$ for 
$H_{N_{m_i}} \subset F_{m(N,m_i)}/2$. Note that, $(G_N)_{N\in \N}$ and $(H_N)_{N\in \N}$ have to be \Folner\ sequences also. 
In order to verify \eqref{eq:4} it 
suffices to show that for every $N\in \N$, 
$$\left| A\cap \Phi_N \right| + \frac{2}{3} \left| \Phi_N\cap \Phi_N/2 \right| \leq \left| \Phi_N \right| + \oh_{N\to\infty}(1),$$
which is implied by
\begin{equation}\label{eq:5}
3\left| A\cap \Phi_N \right| + 2\left| \Phi_N\cap \Phi_N/2 \right| \leq 3 \left| \Phi_N \right| + \oh_{N\to\infty}(1).  
\end{equation}
Now, using the definition of $\Phi_N$ and $A$, we see that the 
left hand side in \eqref{eq:5} equals 
$$ 3|G_N|+\frac{3}{k}|H_N|+2 |G_N/2\cap H_N| + \oh_{N\to\infty}(1),$$
which equals
$$3\sum_{i=1}^{k_N}\left|G_{N_{k_i}}\right| + \frac{3}{k}\sum_{j=1}^{m_N} \left| H_{N_{m_j}} \right| + 
2 \sum_{i,j} \left| G_{N_{k_i}}/2 \cap H_{N_{m_j}}  \right| + \oh_{N\to\infty}(1). $$
Note that we have here used the fact that $(H_N)$ is \Folner\ and the assumption that 
$\{F_N,F_N/2,F_N/4:N\in \N\}$ are all disjoint, because otherwise 
we could have other non-trivial intersections in 
$$\Phi_N\cap \Phi_N/2=(G_N\cup H_N)\cap (G_N/2\cup H_N/2).$$ 
But observe that $G_{N_{k_i}}/2 \cap H_{N_{m_j}} \neq \emptyset$ only if $F_{n(N,k_i)}/2 \cap F_{m(N,m_i)}/2 \neq \emptyset$, i.e. only if $n(N,k_i)=m(N,m_i)$. Note also that in this case, we have $|G_{N_{k_i}}/2 \cap H_{N_{m_j}}|\leq |H_{N_{m_j}}|$. 

For $k\geq 3$ it thus follows that
\begin{align*}
& 3\left| A\cap \Phi_N \right| + 2\left| \Phi_N\cap \Phi_N/2 \right| \leq 3 \sum_{i=1}^{k_N}\left|G_{N_{k_i}}\right| + 3 \sum_{j=1}^{m_N}\left|H_{N_{k_j}}\right| + \oh_{N\to\infty}(1) \\ 
=& 3|G_N|+3|H_N|+\oh_{N\to\infty}(1)=3|\Phi_N|+\oh_{N\to\infty}(1),    
\end{align*}
showing \eqref{eq:5} and therefore the result. 
\end{proof}

Now we prove that Theorem \ref{main theorem} is optimal in 
the sense of Question \ref{open_qu_1}. That is, for 
any countable abelian group, we can find a \Folner\ sequence 
along which some set achieves the threshold density of 
\eqref{main_theorem_1_2} Theorem \ref{main theorem}, without 
containing the sumsets of the statement. This is Theorem \ref{optimality 1 intro} from the introduction.

\begin{theorem}\label{optimality 1}
Let $(G,+)$ be a countable abelian group with 
$\ell=[G:2G]<\infty$. There is a set 
$A\subset G$ such that 
$$\diff_{\Phi}(A)+\diff_{\Phi/2}(A)=1, $$
for some \Folner\ sequence $\Phi$ in $G$, but there does not 
exist an infinite set $B\subset G$ and $t\in G$ with 
$t+B+B\subset A$.    
\end{theorem}

\begin{proof}
In 
\cite[Corollary 5.6]{charamaras_kousek_mountakis_radic2025BBingroups} it was shown that in every countable abelian 
group with $[G:2G]<\infty$ there is a \Folner\ sequence $\Phi$ 
that is not q.i.d. and a set $A\subset \bigcup_{N=1}^{\infty} \Phi_N$ with $\overline{\diff}_{\Phi}(A)=1$ and 
such that $t+B+B \not\subset A$ for any infinite set $B\subset G$ and any element $t\in G$.

It follows by the proof of this corollary, and in particular, the construction in \cite[Lemma 5.5]{charamaras_kousek_mountakis_radic2025BBingroups} that $\Phi$ and $A$ can be chosen so that $\overline{\diff}_{\Phi/2}(A)=0$. Hence, by passing to a subsequence of $\Phi$, which we still write as $\Phi$, we obtain 
$$\diff_{\Phi}(A)+\diff_{\Phi/2}(A)=1,$$
and the proof is complete.
\end{proof}

A similar argument shows the need to consider the 
same subsequence in the densities over the two \Folner\ 
sequences in Theorem \ref{main theorem}, namely Proposition \ref{same subsequence intro}. 

\begin{proposition}\label{same subsequence}
Let $(G,+)$ be a countable abelian group with 
$\ell=[G:2G]<\infty$ and $r=|\ker(D)|<\infty$. There is a set 
$A\subset G$ and a \Folner\ sequence $\Phi$ such that 
$$\overline{\diff}_{\Phi}(A)+\overline{\diff}_{\Phi/2}(A)=2, 
$$
but there does not exist an infinite set $B\subset G$ and 
$t\in G$ with $t+B+B\subset A$.        
\end{proposition}

\begin{proof}
%This is easy to see in $\Z$ (once one is familiar with 
%the construction in 
%\cite[Proposition 4.1]{kousek_radic2025BB}). Indeed, we 
%let 
%$$A=\bigcup_{N=1}^{\infty} [4^{4N}, (2-\frac{1}{N})\cdot 4^{4N}] \cup [(2+\frac{1}{N})\cdot 4^{4N+2}, 4^{4N+3}] \cap \Z ,$$
%we let $\Phi=(\Phi_N)=([4^{4N}, 2\cdot 4^{4N}]\cap \Z)$ and 
%look at $(\Phi_{2N})$ and $(\Phi_{2N+1}/2)$. The details are then easy to verify. 
We again rely on \cite[Corollary 5.6]{charamaras_kousek_mountakis_radic2025BBingroups}. Indeed, we can consider a non-q.i.d. \Folner\ sequence $\Phi=(\Phi_N)$, so that $\Phi/2=(\Phi_N/2)$ is also a non-q.i.d. \Folner\ and even $\Phi_N\cap (\Phi_N/2)=\emptyset$ and $\Phi_N\cap (\Phi_N/4)=\emptyset$. Moreover, we can choose $\Phi=(\Phi_N)$ so that the sets $\Phi_{N}$ are pairwise disjoint and so are the sets $\Phi_N/2$ (see the proof of \cite[Lemma 5.5]{charamaras_kousek_mountakis_radic2025BBingroups} for the construction of a non-q.i.d. \Folner\ ; this is not necessary, we could alternatively refine the sequence so that $|\Phi_{N+1}| \gg |\Phi_N|$). 

Then, we consider the \Folner\ sequence $F=(F_N)$ defined via $F_N=\Phi_{2N} \cup (\Phi_{2N+1}/2)$. Then, $F_N/2=(\Phi_{2N}/2) \cup (\Phi_{2N+1}/4)$ and by the disjointness assumptions it follows that $F=(F_N)$ is non-q.i.d. Passing to a subsequence of $(F_N)$ if needed, we then find a set $A=\bigcup_{N=1}^{\infty} A_N$ so that $A_N\subset F_N$ and $\diff_{F}(A)=1$ and there does not exist an infinite set $B\subset G$ and 
$t\in G$ with $t+B+B\subset A$. Finally, note that 
$$\overline{\diff}_{\Phi}(A)+\overline{\diff}_{\Phi/2}(A)=\diff_{(\Phi_{2N})}(A)+\diff_{(\Phi_{2N+1}/2)}(A)=2.
$$
\end{proof}

Addressing the optimality of Theorem \ref{main theorem} in 
the 
strong sense of Question \ref{open_qu_0} for all abelian 
groups under consideration and all \Folner\ sequences seems 
to be a very difficult task. 
For the moment, we confine ourselves to proving that our 
theorem is optimal in the integer setting for some special 
class of \Folner\ sequences. This is already 
interesting, and useful -- enriching our discussion in Section
\ref{Owing} -- but the construction exploits the 
structure of the group $\Z$ crucially as the reader may 
observe.  

We say that $\Phi=(\Phi_N)$ is a \Folner\ sequence of 
intervals in $\Z$ if each $\Phi_N$ is an interval. In some 
sense, such special 
sequences are the building blocks of all \Folner\ sequences 
in $\Z$. We 
do not make this fact precise, but the 
point is that the generality of the next result (which is Proposition \ref{optimality in Z intro} from the introduction) is 
non-trivial. 

\begin{proposition}\label{optimality in Z}
Let $\Phi$ be any \Folner\ sequence of intervals in $\Z$. 
Then, there 
exists a set $A\subset \Z$ such that 
$\diff_{\Psi}(A)+\diff_{\Psi/2}(A)=1$,
along a subsequence $\Psi$ of $\Phi$ for which the densities exist, for which $t+B+B \not \subset A$, for any infinite set $B$ and 
any $t\in \Z$.
\end{proposition}

\begin{proof}    
We will choose a sequence $(N_m)_{m\in \N}$ and construct a \Folner\ sequence $F=(F_m)_{m\in \N}$ that 
is not q.i.d. for which 
$F_m\subset \Phi_{N_m}\cup \Phi_{N_m}/2$, such that $A'=\bigcup_{m=1}^{\infty} F_m$ satisfies $\overline{\diff}_{\Phi}(A')+\overline{\diff}_{\Phi/2}(A')=1$ and $(N_m)_{m\in \N}$ is chosen so that 
$$\overline{\diff}_{\Phi}(A')+\overline{\diff}_{\Phi/2}(A')=\lim_{m\to \infty} \frac{ |F_m \cap \Phi_{N_m}|}{|\Phi_{N_m}|}+\lim_{m\to \infty} \frac{ |F_m \cap \Phi_{N_m}/2|}{|\Phi_{N_m}/2|}.$$
This can be arranged, for example, by choosing $(N_m)$ so that
$\frac{|\Phi_{N_m}|}{|\Phi_{N_{m+1}}|}\xrightarrow{m\to \infty} 0$.

Having done that, we can then use 
\cite[Proposition 5.4] {charamaras_kousek_mountakis_radic2025BBingroups} in order to 
find a sequence $(m_k)_{k\in \N}$ and set $A\subset \Z$ such 
that 
$A=\bigcup_{k=1}^{\infty} A_{m_k}$, where 
$A_{m_k}\subset F_{m_k}$ and $\overline{\diff}_{F}(A)=\diff_{(F_{m_k})}(A)=1$,
for which $t+B+B \not \subset A$, for any infinite set $B$ 
and 
any $t\in \Z$. These properties of $A$ and the definition of $F$ and $A'$ in turn imply that $\diff_{\Psi}(A)+\diff_{\Psi/2}(A)=1$, where $\Psi=(\Phi_{N_{m_k}})_{k\in \N}$, concluding the proof. 

We are left with constructing $F=(F_m)_{m\in \N}$.
To ease notation, without loss of generality we assume that $N_m=m$, so that (a posteriori) $A'=\bigcup_{N=1}^{\infty}F_{N}$ with $F_N\subset \Phi_N \cup \Phi_N/2$.

For each $N\in \N$, we let $S_{N,0}=\Phi_N\setminus \Phi_N/2$. 
Then, we define $S_{N,1}=S_{N,0}/2^{2}=(\Phi_N/4)\setminus (\Phi_N/8)$, $S_{N,2}=S_{N,0}/2^{4}=(\Phi_N/16)\setminus (\Phi_N/32)$ and more generally
$$S_{N,j}=S_{N,0}/2^{2j}=(\Phi_N/2^{2j})\setminus (\Phi_N/2^{2j+1}),$$
for $j=0,\ldots,k_N$, where $k_N$ is the largest integer $j$ 
for which $\Phi_N/2^{2j} \cap \left(\Phi_N\cup \Phi_N/2\right) \neq \emptyset$. 

We then define $F_N=\bigcup_{j=0}^{k_N} S_{N,j}$. Observe 
that by the definitions, the previous 
union is a disjoint union, and this will be important in the 
calculations to follow. We remark that we have used our 
assumption that $(\Phi_N)$ is a sequence of intervals here. 

Moreover, we observe that $(S_{N,0})$ is a sequence of 
intervals whose lengths increase to infinity and it readily
follows that $(F_N)$ is asymptotically invariant, i.e. a \Folner\ sequence. Finally, 
$$F_N/2=\bigcup_{j=0}^{k_N} S_{N,j}/2=\bigcup_{j=0}^{k_N} (\Phi_N/2^{2j+1})\setminus (\Phi_N/2^{2(j+1)}),$$
so that $F_N\cap F_N/2=\emptyset$. Thus $(F_N)$ is a q.i.d. \Folner\ sequence.  

We then see that 
\begin{align}\label{eq:6} 
& \frac{|F_N\cap \Phi_N|}{|\Phi_N|}=\sum_{j=0}^{k_N} \frac{|S_{N,j}\cap \Phi_N|}{|\Phi_N|}= \frac{|\Phi_N\setminus \Phi_N/2|}{|\Phi_N|} + \sum_{j=1}^{k_N} \frac{|(\Phi_N \cap \Phi_N/2^{2j})\setminus \Phi_N/2^{2j+1} |}{|\Phi_N|}\nonumber \\
=& 1 - \frac{|\Phi_N\cap \Phi_N/2|}{|\Phi_N|} + \sum_{j=1}^{k_N} \frac{|\Phi_N \cap \Phi_N/2^{2j} |}{|\Phi_N|} - \sum_{j=1}^{k_N} \frac{|(\Phi_N \cap \Phi_N/2^{2j}) \cap \Phi_N/2^{2j+1} |}{|\Phi_N|}.
\end{align}
Similarly, but noting that $S_{N,0}\cap \Phi_N/2=\emptyset$, 
and since $|H_N/2|=|H_N|/2+\oh_{N\to\infty}(1)$, for any \Folner\ sequence $(H_N)$ in $\Z$, we get
that 
\begin{equation}\label{eq:7}
\frac{|F_N\cap \Phi_N/2|}{|\Phi_N/2|}=\sum_{j=1}^{k_N} \frac{|\Phi_N \cap \Phi_N/2^{2j-1} |}{|\Phi_N|} - \sum_{j=1}^{k_N} \frac{|(\Phi_N \cap \Phi_N/2^{2j-1} ) \cap \Phi_N/2^{2j} |}{|\Phi_N|}+\oh_{N\to\infty}(1). 
\end{equation}
Now, observe that $\Phi_N \cap \Phi_N/2^i \cap \Phi_N/2^{i+1} = \Phi_N \cap \Phi_N/2^{i+1}$, where we have crucially used that $(\Phi_N)$ is a sequence of intervals once again. Using this in \eqref{eq:6} and \eqref{eq:7} we obtain that 
\begin{equation}\label{eq:8}
\frac{|F_N\cap \Phi_N|}{|\Phi_N|} = 1 - \frac{|\Phi_N\cap \Phi_N/2|}{|\Phi_N|} + \sum_{j=1}^{k_N} \frac{|\Phi_N \cap \Phi_N/2^{2j} |}{|\Phi_N|} - \sum_{j=1}^{k_N} \frac{|\Phi_N \cap \Phi_N/2^{2j+1} |}{|\Phi_N|}
\end{equation}
and 
\begin{equation}\label{eq:9}
\frac{|F_N\cap \Phi_N/2|}{|\Phi_N/2|} = \sum_{j=1}^{k_N} \frac{|\Phi_N \cap \Phi_N/2^{2j-1} |}{|\Phi_N|} - \sum_{j=1}^{k_N} \frac{|\Phi_N \cap \Phi_N/2^{2j} |}{|\Phi_N|}+\oh_{N\to\infty}(1). 
\end{equation}
It follows that 
$$\frac{|F_N\cap \Phi_N|}{|\Phi_N|} + \frac{|F_N\cap \Phi_N/2|}{|\Phi_N/2|}= 1+\oh_{N\to\infty}(1),$$
where we also used that $\Phi_N \cap \Phi_N/2^{2k_N+1} = \emptyset$. Indeed, if $\Phi_N \cap \Phi_N/2^{2k_N+1} \neq \emptyset$ then we would also have $\Phi_N/2 \cap \Phi_N/2^{2(k_N+1)} \neq \emptyset$, contradicting the definition of $k_N$. Therefore, 
\begin{equation*}
\diff_{\Phi}(A')+ \diff_{\Phi/2}(A') = 1
\end{equation*}
and this concludes the proof.
\end{proof}

\appendix 

\section{Connection to Owings' problem}\label{Owing}

We now present a connection between the main result and 
Owings' problem mentioned in the introduction. The 
observations presented herein merely serve the purpose of 
illustrating one of the reasons why our main result, Theorem \ref{main theorem}, was discovered -- as explained in the introduction. Indeed, Owings' problem was solved in the positive very recently in \cite{owings_solution}. Perhaps one conclusion of the considerations to follow is, in our opinion, that there probably is no density-type of argument solution to the problem. 

\begin{question}\label{B+B}
Let $\N=C_1 \cup C_2$ be a $2$-colouring of 
the natural numbers. Does there exist an infinite set $B=\{b_1<b_2< \cdots \}\subset \N$, such that $B+B=\{b_i+b_j: i,j\in \N\} \subset C_{k}$, for some $k\in \{1,2\}$?
\end{question}

As was observed in \cite[Remark 6.2]{kousek_radic2025BB}, in 
order to answer Question \ref{B+B}, one may allow a shift in 
the sumsets. This follows directly from the fact that $\N=C_1/2 \cup C_2/2$, where $C_k/2:=\{n\in \N: 2n\in C_k\}$, is also a $2$-colouring of $\N$, and the observation that $B+B+t \subset C_k/2$ implies that $(2B+t)+(2B+t) \subset C_k.$ Specifically, the following question is 
equivalent to Question \ref{B+B}.

\begin{question}\label{B+B shift}
Let $\N=C_1 \cup C_2$ be a $2$-colouring of 
the natural numbers. Does there exist an infinite set $B=\{b_1<b_2< \cdots \}\subset \N$, and some $t\in \N_0$, such that $t+B+B=\{t+b_i+b_j: i,j\in \N\} \subset C_{k}$, for some $k\in \{1,2\}$?
\end{question}

We repeat that Questions \ref{B+B} and \ref{B+B shift} both have a positive answer, as shown in \cite{owings_solution}.  

\begin{theorem}[{\cite[Theorem 1.1]{owings_solution}}]  \label{owings_solution_theorem}
Let $\N=C_1\cup C_2$. Then there exists $i\in \{1,2\}$ and an infinite $B\subset \N$ such that $B+B\subset C_i$.   
\end{theorem}

As an immediate corollary of Theorem \ref{main theorem} for 
the integers we obtain the result in the natural numbers as 
well. 

\begin{corollary} \label{B+B in density}
Let $A\subset \N$ be a set and $\Phi=(\Phi_N)_{N\in \N}$ be a \Folner\ sequence in $\N$. Then, if 
\begin{equation}\label{density condition}
\d_{\Phi}(A)+\d_{\Phi/2}(A)>1,
\end{equation}
it follows that there exists an infinite set $B\subset \N$ and some $t\in \N_0$ such that $B+B \subset A-t$.
\end{corollary}

The relevant observation is that if Corollary \ref{B+B in density} (or, Theorem \ref{main theorem} in the integers) was not optimal, then we 
would obtain answer to Question \ref{B+B} as a direct application of this. More precisely, assume 
there exists a \Folner\ sequence $\Phi$ in $\N$ and some $\epsilon>0$ such that 
\eqref{density condition} in the context of 
Corollary \ref{B+B in density} can be replaced by
$\d_{\Phi}(A)+\d_{\Phi/2}(A)>1-\epsilon$, yielding the same conclusion. 
Then, given any colouring $\N=C_1\cup C_2$, and passing to a subsequence of $\Phi=(\Phi_N)$ if necessary, it must be the 
case that 
$$\d_{\Phi}(C_1)+\d_{\Phi/2}(C_1)+\d_{\Phi}(C_2)+\d_{\Phi/2}(C_2) =2,$$
which implies that for some $k\in \{1,2\}$, $\d_{\Phi}(C_k)+\d_{\Phi/2}(C_k)>1=\epsilon$, and thus there exist infinite $B\subset \N$ and $t\in \N$ such that $B+B \subset C_k-t$.

We have already showed that Corollary \ref{B+B in density} is optimal along \Folner\ sequences of intervals and we believe the proof should extend to all \Folner\ sequences. In particular, this would mean that an argument as the one laid out just now cannot be used to answer Question \ref{B+B shift}, i.e. to given an alternative proof of Theorem \ref{owings_solution_theorem}.

\bibliographystyle{abbrv}
\bibliography{Refs}

\vspace{1cm}

\end{document}